\documentclass[11pt,reqno,tbtags]{amsart}
\pdfoutput=1
\usepackage[]{amsmath,amssymb,amsfonts,latexsym,amsthm,enumerate,xcolor,tikz}
\usepackage{amssymb,bbm,mathrsfs}
\usepackage{enumerate,graphicx,paralist}
\usepackage[numeric,initials,nobysame]{amsrefs}

\numberwithin{equation}{section}
\usepackage[colorlinks=true, pdfstartview=FitV, linkcolor=blue, citecolor=blue, urlcolor=blue,pagebackref=false]{hyperref}
\usepackage[marginratio=1:1,height=584pt,width=488pt,tmargin=117pt]{geometry}
\usepackage[margin=25pt,font=small,justification=centering
]{caption}

\newtheorem{maintheorem}{Theorem}

\newtheorem{mainprop}[maintheorem]{Proposition}

\newtheorem{theorem}{Theorem}[section]
\newtheorem*{theorem*}{Theorem}
\newtheorem{lemma}[theorem]{Lemma}
\newtheorem{claim}[theorem]{Claim}

\theoremstyle{definition}{

\newtheorem*{definition*}{Definition}

\newtheorem*{example*}{Example}

\newtheorem*{remark*}{Remark}

}

\newcommand\ceil[1]{\lceil#1\rceil}

\newcommand{\R}{\mathbb R}

\newcommand{\Z}{\mathbb Z}

\newcommand{\E}{\mathbb{E}}
\renewcommand{\P}{\mathbb{P}}
\DeclareMathOperator{\var}{Var}
\DeclareMathOperator{\Cov}{Cov}

\DeclareMathOperator{\diam}{diam}

\newcommand{\tmix}{t_\textsc{mix}}

\newcommand{\tv}{{\textsc{tv}}}

\newcommand{\one}{\mathbbm{1}}

\renewcommand{\epsilon}{\varepsilon}
\renewcommand{\phi}{\varphi}

\newcommand{\cG}{\mathcal{G}}

\newcommand{\cA}{\mathcal{A}}
\newcommand{\cB}{\mathcal{B}}

\newcommand{\cD}{\mathcal{D}}
\newcommand{\cE}{\mathcal{E}}

\newcommand{\cI}{\mathcal{I}}
\newcommand{\cJ}{\mathcal{J}}
\newcommand{\cK}{\mathcal{K}}

\newcommand{\ltwo}{{\mathfrak{M}}}

\newcommand{\sH}{\mathscr{H}}

\newcommand{\fsup}{\mathscr{F}_\textsc{s}}

\newcommand{\tpluss}{t_\star}

\newcommand{\tsups}{t_{-}}
\newcommand{\alt}{\mathrm{alt}}
\newcommand{\blt}{\mathrm{blt}}

\date{}

\begin{document}
\title{Fast Initial Conditions for Glauber Dynamics }

\author{Eyal Lubetzky}
\address{Eyal Lubetzky\hfill\break
Courant Institute of Mathematical Sciences
\\ New York University\\
New York, NY 10012, USA.}
\email{eyal@courant.nyu.edu}
\urladdr{}

\author{Allan Sly}
\address{Allan Sly\hfill\break
Department of Mathematics\\
Princeton University\\
Princeton, NJ 08544, USA, and\hfill\break
Department of Statistics\\
UC Berkeley\\
Berkeley, CA 94720, USA.
}
\email{asly@math.princeton.edu}
\urladdr{}

\begin{abstract}
In the study of Markov chain mixing times, analysis has centered on the performance from a worst-case starting state.  Here, in the context of Glauber dynamics for the one-dimensional Ising model, we show how new ideas from information percolation can be used to establish mixing times from other starting states.  At high temperatures we show that the alternating initial condition is asymptotically the fastest one, and, surprisingly, its mixing time is faster than at infinite temperature, accelerating as the inverse-temperature $\beta$ ranges from 0 to $\beta_0=\frac12\mathrm{arctanh}(\frac13)$. Moreover, the dominant test function depends on the temperature:  at $\beta<\beta_0$ it is autocorrelation, whereas at $\beta>\beta_0$ it is the Hamiltonian.
\end{abstract}
\maketitle
\vspace{-0.75cm}

\section{Introduction}

In the study of mixing time of Markov chains, most of the focus has been on determining the asymptotics of the worst-case mixing time, while relatively little is known about the relative effect of different initial conditions.  The latter is quite natural from an algorithmic perspective on sampling, since one would ideally initiate the dynamics from the fastest initial condition. However, until recently, the tools available for analyzing Markov chains on complex systems, such as the Ising model, were insufficient for the purpose of comparing the effect of different starting states; indeed, already pinpointing the asymptotics of the worst-case state for Glauber dynamics for the Ising model can be highly nontrivial.

In this paper we compare different initial conditions for the Ising model on the cycle.  In earlier work~\cite{LS4}, we analyzed three different initial conditions. The all-plus state is provably the worst initial condition up to an additive constant.  Another is a quenched random condition chosen from $\nu$, the uniform distribution on configurations, which with high probability has a mixing time which is asymptotically as slow.  A third initial condition is an annealed random condition chosen from $\nu$, i.e., to start at time 0 from the uniform distribution, which is asymptotically twice as fast as all-plus.

Here we consider two natural deterministic initial configurations.  The first is the alternating sequence
\[
x^{\alt}(i) = \begin{cases}
1 &i\equiv 0  \!\!\pmod 2 \\
-1 &i\equiv 1   \!\!\pmod 2  \, ,
\end{cases}
\]
which we will show is asymptotically the \emph{fastest} deterministic initial condition---yet strictly slower than starting from the annealed random condition---for all $\beta < \beta_0 := \frac12\operatorname{arctanh}(\frac13)$ (at  $\beta=\beta_0$ they match).
The second is the bi-alternating sequence
\[
x^{\blt}(i) = \begin{cases}
1 &i\equiv 0,3  \!\!\pmod 4 \\
-1 &i\equiv 1,2  \!\! \pmod 4  \, .
\end{cases}
\]
For convenience we will assume that $n$ is a multiple of 4, which ensures that the configurations are semi-translation invariant and turns both sequences into  eigenvectors of the transition matrix of simple random walk on the cycle.  (This is not necessary for the main result but leads to cleaner analysis.)

In what follows, set $\theta=\theta_{\beta}=1 - \tanh(2\beta)$, and let $\tmix^{x_0}(\epsilon)$ denote the time it takes the dynamics to reach total variation distance at most $\epsilon$ from stationarity, starting from the initial condition $x_0$.

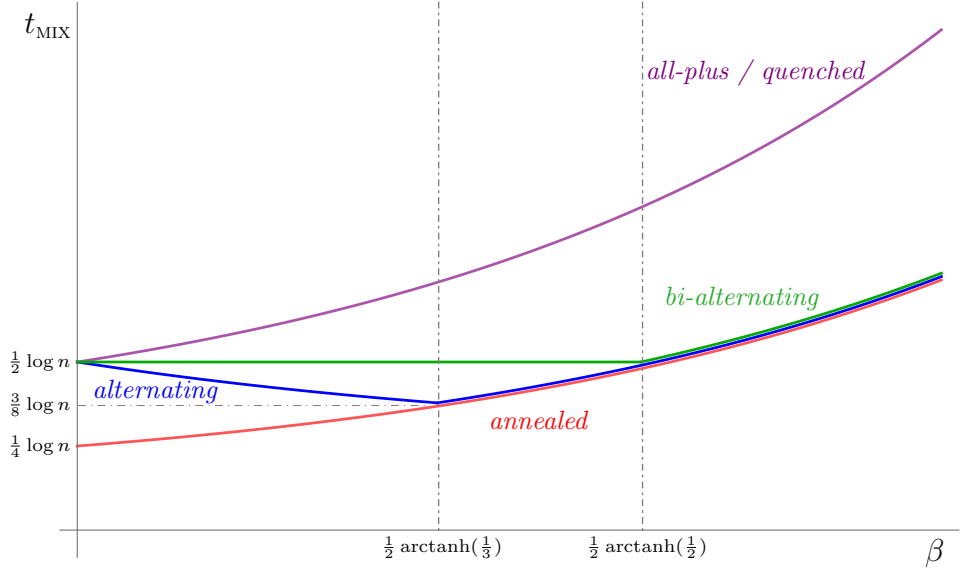
\begin{figure}
\vspace{-0.175in}
  \hspace{-0.15in}
  \begin{tikzpicture}
    \node (plot) at (0,0)
    {\includegraphics[width=0.7\textwidth]{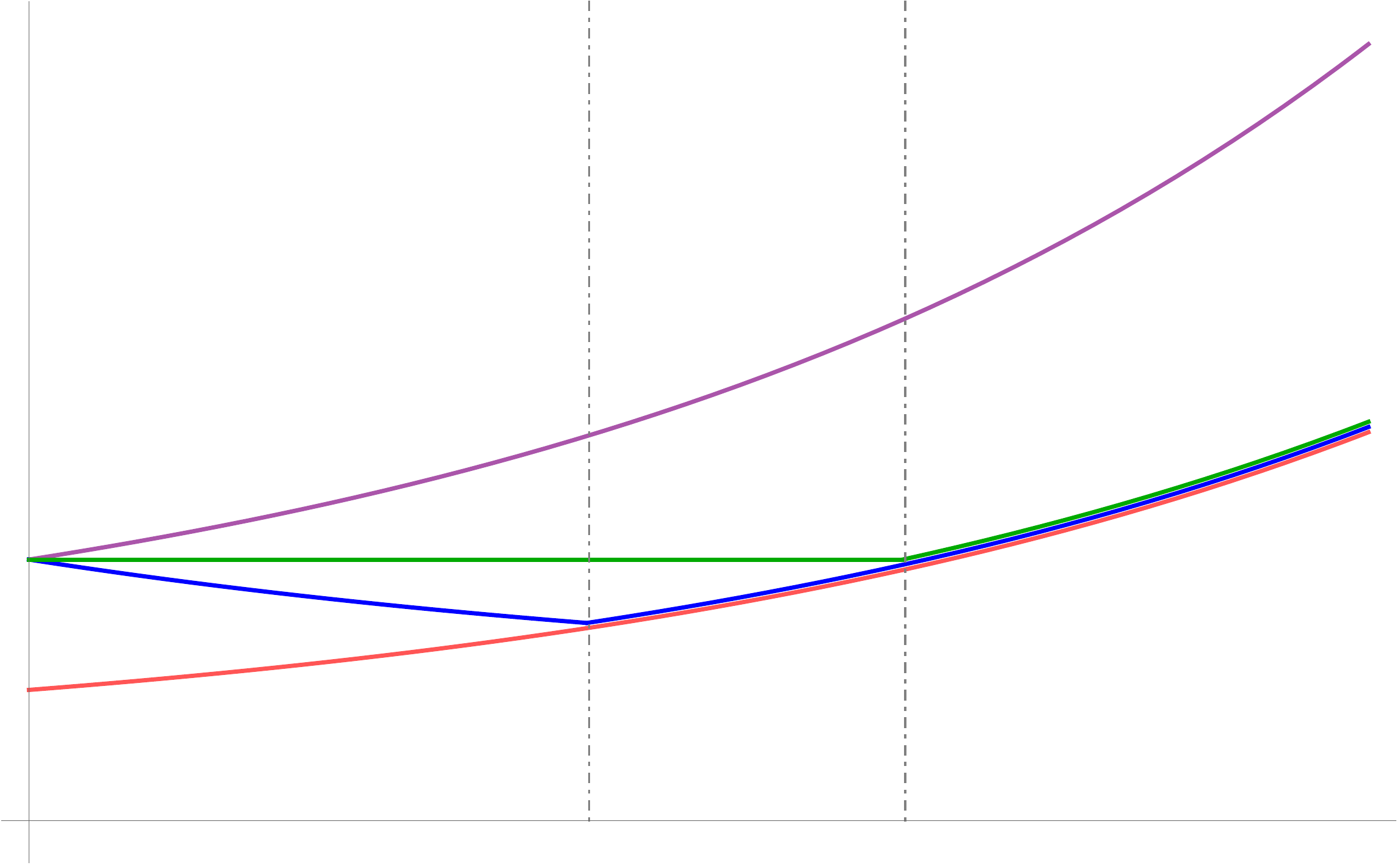}};

    \begin{scope}[shift={(plot.south west)},x={(plot.south
        east)},y={(plot.north west)}, font=\small]

      \node[font=\tiny] at (-0.01,0.215) {$\frac14\log n$};
      \node[font=\tiny] at (-0.01,0.36) {$\frac12\log n$};
      \node[font=\large] at (0,0.94) {$\tmix$};

      \node[font=\tiny] at (-0.01,0.286) {$\frac38\log n$};
      \draw[color=gray,dashdotted] (0.032,0.285) -- (0.41,0.285);

      \node[font=\tiny] at (0.426,0.04) {$\frac12\operatorname{arctanh}(\frac13)$};
      \node[font=\tiny] at (0.648,0.04) {$\frac12\operatorname{arctanh}(\frac12)$};
      \node[font=\large] at (.96,0.03) {$\beta$};

      \node[color={rgb:red,0.5;green,0;blue,0.5}] at (0.765,0.86) {{\it all-plus / quenched}};
      \node[color={rgb:red,0.16;green,0.66;blue,0.16}] at (0.75,0.47) {{\it bi-alternating}};
      \node[color=red] at (0.53,0.26) {{\it annealed}};
      \node[color=blue] at (0.115,0.31) {{\it alternating}};

    \end{scope}
  \end{tikzpicture}
  \vspace{-0.225in}
  \caption{Asymptotic mixing time from the alternating and bi-alternating initial conditions as per Theorem~\ref{mainthm-trans}, compared to the  known behavior of worst-case (all-plus) and random initial conditions.}
  \vspace{-0.075in}
  \label{fig:initial}
\end{figure}

\begin{maintheorem}\label{mainthm-trans}
For every $\beta>0$ and $0<\epsilon<1$ there exist $C(\beta)$ and $N(\beta,\epsilon)$ such that the following hold for Glauber dynamics for the Ising model on the cycle $\Z/n\Z$ at inverse-temperature $\beta$ for all $n>N$.
\begin{enumerate}[(i)]
\item Alternating initial condition:
\[
\left|\tmix^{x^{\alt}}(\epsilon) - \max\big\{\tfrac{1}{4-2\theta},\tfrac{1}{4\theta}\big\}\log n \right| \leq C \log \log n\,.
\]
\item Bi-alternating initial condition:
\[
\left|\tmix^{x^{\blt}}(\epsilon) - \max\left\{\tfrac{1}{2},\tfrac{1}{4\theta}\right\}\log n \right| \leq C \log \log n\,.
\]
\end{enumerate}
\end{maintheorem}
Surprisingly, the mixing time for the alternating initial condition begins as actually \emph{faster} than the infinite temperature model: it \emph{decreases} as a function of $\beta$ before increasing when $\beta > \frac12 \operatorname{arctanh}(\frac13)$.

The following theorem summarizes the bounds we proved in~\cite{LS1,LS4} for the all-plus and random initial conditions.
See Figure~\ref{fig:initial} for the relative performance of all these different initial conditions.

\begin{maintheorem}[\cite{LS1,LS4}]\label{mainthm-previous-bounds} In the same setting of Theorem~\ref{mainthm-trans}, the following hold.
\begin{enumerate}[(i)]
\item All-plus initial condition $x^+ \equiv 1$:
\[
\left|\tmix^{x^+}(\epsilon) - \tfrac{1}{2\theta}\log n \right| \leq C \log \log n \, .
\]
\item Quenched random initial condition:
\[
\nu\left(\left\{ x_0: \left|\tmix^{x_0}(\epsilon) - \tfrac{1}{2\theta}\log n \right| \leq C \log \log n \right\}\right) \to 1  \quad\mbox{as $n\to\infty$}\,.
\]
\item Annealed random initial condition:
\[
\left|\tmix^{\nu}(\epsilon) - \tfrac{1}{4\theta}\log n \right| \leq C \log \log n \, .
\]
\end{enumerate}
\end{maintheorem}

(Note that, in the case of the all-plus initial conditions,  the mixing time $\tmix^{x^+}(\epsilon)$  is known in higher precision: it was shown~\cite{LS1,LS4} to be  within an additive constant (depending on $\epsilon$ and $\beta$) of $\frac{1}{2\theta}\log n$.)

The upper bounds on the mixing times in Theorem~\ref{mainthm-trans} rely on the information percolation framework  introduced by the authors in~\cite{LS4}. The asymptotically matching lower bounds in that theorem are derived from two test functions: the autocorrelation function, which for instance matches our upper bound on the alternating initial condition for $\beta>\beta_0$; and the Hamiltonian test function, which gives rise to the following lower bound on \emph{every deterministic initial condition}.

\begin{mainprop}\label{p:autoLowerB}
Let $X_t$ be Glauber dynamics for the Ising model on  $\Z/n\Z$ at inverse-temperature $\beta$.
For every sequence of deterministic initial conditions $x_0$, the dynamics  at time
 \[ \tpluss =  \tfrac{1}{4-2\theta}\log n - 8\log\log n\]
is at total variation distance $1-o(1)$ from equilibrium; that is,
\[
\lim_{n\to\infty} \inf_{x_0} \left\|\P_{x_0}\left(X_{\tpluss} \in \cdot\right) - \pi \right\|_\tv = 1\,.
\]
\end{mainprop}

As a consequence of this result and Theorem~\ref{mainthm-trans}, Part~(i), we see that the initial condition $x^\alt$ is indeed the optimal \emph{deterministic} one in the range $\beta < \beta_0$, and that $\beta_0$ marks the smallest  $\beta$ where a deterministic initial condition can first match the performance of the annealed random condition.

The mixing time estimates in Theorem~\ref{mainthm-trans} (as well as those in Theorem~\ref{mainthm-previous-bounds}) imply, in particular, that Glauber dynamics for the Ising model on the cycle, from the respective starting configurations, exhibits the \emph{cutoff phenomenon}---a sharp transition in its distance from stationarity, which drops along a negligible time period known as the \emph{cutoff window} (here, $O(\log\log n)$, vs.\ $\tmix$  which is of order $\log n$) from near its maximum to near 0. Until recently, only relatively few occurrences of this phenomenon, that was discovered by Aldous and Diaconis in the early 1980's (see~\cites{Aldous,AD,DiSh,Diaconis}), were rigorously verified,
even though it is believed to be widespread (e.g., Peres conjectured~\cite{LLP}*{Conjecture~1},\cite{LPW}*{\S23.2} cutoff for the Ising model on any sequence of transitive graphs when the mixing time is of order $\log n$); see~\cite{LPW}*{\S18}.

For the Ising model on the cycle, the longstanding lower and upper bounds on $\tmix$ from a worst-case initial condition differed by a factor of 2---in our notation, $\frac{1-o(1)}{2\theta}\log n$ and $\frac{1+o(1)}\theta\log n$---while cutoff was conjectured to occur (see, e.g.,~\cite{LPW}*{Theorem~15.4}, as well as~\cite{LPW}*{pp.~214,248 and Question~8 in p.~300}). This was confirmed in~\cite{LS1}, where the above lower bound was shown to be tight, via a proof that relied on log-Sobolev inequalities and applied to $\Z^d$, for any dimension $d\geq 1$, so long as the system features a certain decay-of-correlation property known as strong spatial mixing. This result was reproduced in~\cite{LS4} (with a finer estimate for the cutoff window) via the new information percolation method. Soon after, a remarkably short proof of cutoff for the cycle---crucially hinging on the correspondence between the one-dimensional Ising model and the ``noisy voter'' model---was obtained by Cox, Peres and Steif~\cite{CPS}. It is worthwhile noting that the arguments both in~\cite{CPS} and in~\cite{LS1} are tailored to worst-case analysis, and do not seem to be able to treat specific initial conditions as examined here. In contrast, the information percolation approach does allow one to control the subtle effect of various   initial conditions on mixing.

To conclude this section, we conjecture that Proposition~\ref{p:autoLowerB} also holds for $\tpluss =\max\{\frac{1-o(1)}{4-2\theta},\frac{1-o(1)}{4\theta}\}\log n$, i.e., that $x^\alt$ is asymptotically fastest among all the deterministic initial conditions at all $\beta>0$.
We further conjecture that the obvious generalization of $x^\alt$ to $(\Z/n\Z)^d$ for $d\geq 2$ (a checkerboard for $d=2$) is the analogous fastest deterministic  initial condition  throughout the high-temperature regime.

\section{Update support and information percolation}\label{sec:upperBound}

In this section we define the update support and use the framework of information percolation (see the papers~\cite{LS3,LS5} as well as the survey paper~\cite{LS6} for an exposition of this method)  to upper bound the total variation distance with alternating and bi-alternating initial conditions.

\subsection{Basic Notation}

The Ising model on a finite graph $G$ with vertex-set $V$ and edge-set $E$ is a distribution over the set of configurations $\Omega=\{\pm1\}^V$;
each $\sigma\in\Omega$ is an assignment of plus/minus \emph{spins} to the sites in $V$, and the probability of $\sigma \in \Omega$ is given by the Gibbs distribution
\begin{equation}
  \label{eq-Ising}
  \pi(\sigma)  = \mathcal{Z}^{-1} e^{\beta \sum_{uv\in E} \sigma(u)\sigma(v) } \,,
\end{equation}
where $\mathcal{Z}$ is a normalizer (the partition-function) and $\beta$ is the inverse-temperature, here taken to be non-negative (ferromagnetic).
The (continuous-time) heat-bath Glauber dynamics for the Ising model is the Markov chain---reversible w.r.t.\ the Ising measure $\pi$---where each site is associated with a rate-1 Poisson clock, and as the clock at some site $u$ rings, the spin of $u$ is replaced by a sample from the marginal of $\pi$ given all other spins. See~\cite{Martinelli97} for an extensive account of  this dynamics. In this paper we focus on the graph $G=\Z/n\Z$ and will let $X_t$ denote the Glauber dynamics Markov chain on $G$.

An important notion of measuring the convergence of a Markov chain $(X_t)$ to its stationarity measure $\pi$ is its total-variation mixing time, denoted $\tmix(\epsilon)$ for a precision parameter $0<\epsilon<1$.  From initial condition $x_0$ we denote
\[ \tmix^{x_0}(\epsilon) = \inf\big\{t \;:\;  \| \P_{x_0}(X_t \in \cdot)- \pi\|_\tv \leq \epsilon \big\}\,,\]
and the overall mixing time as measured from a worst-case initial condition is
\[ \tmix(\epsilon) = \max_{x_0 \in \Omega}\tmix^{x_0}(\epsilon) \,,\]
where here and in what follows $\P_{x_0}$ denotes the probability given $X_0=x_0$, and the total-variation distance
$\|\mu_1-\mu_2\|_\tv$ is defined as $\max_{A\subset \Omega} |\mu_1(A)-\mu_2(A)| = \tfrac12\sum_{\sigma\in\Omega} |\mu_1(\sigma)-\mu_2(\sigma)|$.

\subsection{Information percolation clusters}
The dynamics can be viewed as a deterministic function of $X_0$ and a random ``update sequence'' of the form $(J_1,U_1,t_1),(J_2,U_2,t_2),\ldots$, where $0<t_1<t_2<\ldots$ are the update times (the ringing of the Poisson clocks), the $J_i$'s are i.i.d.\ uniformly chosen sites (which clocks ring), and the $U_i$'s are i.i.d.\ uniform variables on $[0,1]$ (to generate coin tosses).  There are a variety of ways to encode such updates but in the case of the one-dimensional model there is a particularly useful one.  We add an extra variable $S_i$ which is a randomly selected neighbor of $U_i$ Then given the sequence of $(J_i, S_i, U_i, t_i)$  the updates are processed sequentially as follows: set $t_0=0$; the configuration $X_t$ for all $t\in [t_{i-1},t_i)$ ($i\geq 1$) is obtained by updating the site $J_i$ via the unit variable as follows: if $U_i \leq \theta=1 - \tanh(2\beta)$ update the spin at $J_i$ to a  uniformly random value and with probability $1-\theta$ set it to the spin of $S_i$.

With this description of the dynamics, we can work backwards to describe how the configurations at time $\tpluss$ (or at any intermediate time) depend on the initial condition.
\emph{The update support function}, denoted $\fsup(A,s_1,s_2)$, as introduced in~\cite{LS1}, is the random set whose value is  the minimal subset $S\subset \Lambda$ which determines the spins of $A$ given the update sequence  along the interval $(s_1,s_2]$.

We now describe the support of a vertex $v\in V$ as it evolves backwards in time from $s_2$ to $s_1$.
Initially, $\fsup(v,s_2,s_2)=\{v\}$; then,  updates in reverse chronological order alter the support: given the next update $(J_i, S_i, U_i, t_i)$, if $J_i = \fsup(v,t_{i+1},s_2)$ and $U_i \leq \theta$ then $\fsup(v,t_{i},s_2)$ is set to $ \emptyset$, and if $U_i > \theta$ then it is set to $ S_i$. Thus, backwards in time $\fsup(v,t,s_2)$ performs a continuous-time simple random walk with jump rate $1-\theta$ which is killed at rate $\theta$.  We refer to the full trajectory of the update support of a vertex as the \emph{history} of the vertex.  The survival time for a walk is exponential and so for $t_1 \leq t_2$,
\begin{equation}\label{eq:survivalProb}
\P\left(\fsup(v,t_1,t_2) \neq \emptyset\right) = e^{-(t_2-t_1) \theta}\,.
\end{equation}
For general sets $A$ we have that $\fsup(A,s_1,s_2)= \bigcup_{v\in A} \fsup(v,s_1,s_2)$ and taken together the collection of the update supports of the vertices are a set of coalescing killed continuous-time random walks.

A key use of these histories is to effectively bound the spread of information, as achieved by the following lemma.
\begin{lemma}\label{l:spread}
For any $t$ we have that
\[
\P \bigg(\max_{v\in \Z/ n\Z} \max_{\substack{0\leq s \leq t\\ \fsup(v,s,t) \neq \emptyset}} |v - \fsup(v,s,t)| \geq \tfrac1{10} \log^2 n\bigg) \leq O(n^{-10})\,.
\]
\end{lemma}
\begin{proof}
By equation~\eqref{eq:survivalProb} we have that $\P[\fsup(\Z/n\Z , t-\log^{3/2} n, t) \neq \emptyset] = O(n^{-10})$ so it is sufficient to show that
\[
\P \bigg(\max_{\substack{t-\log^{3/2} n \leq s \leq t\\ \fsup(v,s,t) \neq \emptyset}} |v - \fsup(v,s,t)| \geq \tfrac1{10} \log^2 n\bigg) \leq O(n^{-11})\,.
\]
This probability is bounded above by the probability of a rate $1-\theta$ continuous-time random walk to make at least $\frac1{10} \log^2 n$ jumps by time $\log^{3/2} n$.  This is exactly the probability that a Poisson with mean $(1-\theta) \log^{3/2} n$ is at least $\frac1{10} \log^2 n$, which satisfies the required bound by standard tail bounds.
\end{proof}

\section{Upper bounds}

We will consider the dynamics run up to time $\tpluss$ and derive an upper bound on its mixing time.  We will first estimate the total variation distance not of the full dynamics but simply at a single vertex from initial conditions $x^{\alt}$ and $x^{\blt}$.
\begin{lemma}\label{l:singlePoint}
For $v\in \Z/n\Z$ we have that,
\begin{align*}
\left\|\P_{x^{\alt}}\left(X_{\tpluss}(v) \in \cdot \right) - \pi|_v \right\|_\tv &= \tfrac12 e^{-(2-\theta)\tpluss}\,,\\
\left\|\P_{x^{\blt}}\left(X_{\tpluss}(v) \in \cdot \right) - \pi|_v \right\|_\tv &= \tfrac12 e^{- \tpluss}\,.
\end{align*}
\end{lemma}
\begin{proof}
We will begin with the case of initial condition $x^{\alt}$.  Of course $\pi|_v$ is the uniform measure on~$\{\pm 1\}$.  The history $\fsup(v,t,\tpluss)$ is killed before time $0$ with probability $1-e^{-\theta \tpluss}$ and on this event is uniform on $\{\pm 1\}$.  Condition that it survives to time $0$ and let $Y(s) = x^{\alt}(\fsup(v,\tpluss-s,\tpluss))$.  This is simply a continuous-time random walk on $\{\pm 1\}$ which switches state at rate $1-\theta$.  Thus,
\[
\P\left(Y(s) = a\right) = \begin{cases}
                 \frac12 + \frac12 e^{-2(1-\theta)s} & \mbox{if } a = x^{\alt}(v)\,, \\
                 \frac12 - \frac12 e^{-2(1-\theta)s} & \mbox{otherwise}\, .
               \end{cases}
\]
It therefore follows that $\left\|\P\left(Y(\tpluss) \in \cdot\right) - \pi|_v \right\|_\tv = \frac12 e^{-2(1-\theta)\tpluss}$, and altogether,
\[
\left\|\P_{x^{\alt}}\left(X_{\tpluss}(v) \in \cdot \right) - \pi|_v \right\|_\tv =  \tfrac12  e^{-2(1-\theta)\tpluss} e^{-\theta \tpluss} = \tfrac12 e^{-(2-\theta)\tpluss}\,.
\]
The case of $x^{\blt}$ follows similarly, with the exception that $Y(s)$ has jump rate $\frac12(1-\theta)$ since it only switches sign with probability $\frac12$ each step.
\end{proof}

\subsection{Update Support}
In this subsection we analyse the geometry of the update support similarly to~\cite{LS1} in order to approximate the Markov chain as a product measure.  Let $\kappa=\frac{4}{1-\theta}$ and define the support time as $\tsups = \tpluss - \kappa \log \log n$.  By Lemma~\ref{l:spread} we expect the histories to not travel ``too far'' along the time-interval $\tpluss$ to $\tsups$; precisely, if we  define $\cB$ as the event
\[
\cB=\bigg \{ \max_{v\in \Z / n \Z} \max_{\substack{\tsups\leq s \leq \tpluss\\ \fsup(v,s,\tpluss) \neq \emptyset}} |v - \fsup(v,s,\tpluss)| \leq \tfrac1{10} \log^2 n \bigg\}\,,
\]
then by Lemma~\ref{l:spread},
\begin{equation}\label{eq:cBBound}
\P\left(\cB\right) \geq 1 - n^{-10}\,.
\end{equation}
The following event says that the support at time $\tsups$ clusters into small well separated components.
Let $\mathcal{A}$ be the event that there exists a set of intervals $W_1,\ldots,W_m \subset \Z / n\Z$  that (i) cover the support:
\begin{equation}\label{eq-intervals-prop1}
\left\{x: \fsup(x, \tsups,\tpluss) \neq \emptyset\right\} \subset \bigcup_{i} W_i\,,	
\end{equation}
(ii) have logarithmic size:
\begin{equation}\label{eq-intervals-prop2}
\max_i W_i \leq \log^3 n\,,
\end{equation}
and (iii) are well-separated:
\begin{equation}\label{eq-intervals-prop2}
\min_{i,i'} \  d(W_i,W_{i'}) \geq  \log^2 n\,.
\end{equation}
\begin{lemma}
We have that $\P\left(\cA\right) \geq 1 - O(n^{-9})$.
\end{lemma}
\begin{proof}
Define the following intervals on $\Z/n\Z$:
\[
M_i  = \{2i \log^2 n, \ldots, (2i+1) \log^2 n\}\qquad(1\leq i \leq \tfrac{n}{2\log^2 n})\,.\]
  Restricting $\cB$ to $\bigcup M_i$, we let
\[
\cB'=\biggl \{ \max_{v\in \cup_i M_i}  \max_{\substack{\tsups\leq s \leq \tpluss\\ \fsup(v,s,\tpluss) \neq \emptyset}} |v - \fsup(v,s,\tpluss)| \leq \tfrac1{10} \log^2 n \biggr\}\,.
\]
Since $\cB' \supset \cB$ we have that $\P\left(\cB'\right)  \geq 1 - n^{-10}$ by Lemma~\ref{eq:cBBound}. Next, let $\cD_i$ be the event
\[
\cD_i = \{ \fsup(M_i,\tsups,\tpluss) = \emptyset \}\,.
\]
By a union bound and equation~\eqref{eq:survivalProb}, we have that
\[
\P\left(\cD_i\right) \geq 1 - |M_i| e^{ (1-\theta) \kappa \log \log n} \geq 1 - \frac1{\log n}\,,
\]
and so
\[\P\left(\cD_i^c \mid \cB'\right) \leq \frac{\P\left(\cD_i^c\right)}{\P\left(\cB'\right)} \leq \frac2{\log n}\,.\]  Moreover, conditional on $\cB'$ the events $\cD_i$ are conditionally independent since the history of $M_i$ is determined by the updates within the set $\{v:d(v,M_i) \leq \frac1{10} \log^2 n\}$ which are disjoint.  Hence, for all $i$,
\[
\P\left(\cD_i^c,\cD_{i+1}^c,\ldots, \cD_{i+\frac1{10}\log n}^c\mid \cB'\right) \leq \left(\frac2{\log n}\right)^{\frac1{10} \log n} \leq n^{-10};
\]
hence,
\[
\P\left(\cD_i^c,\cD_{i+1}^c,\ldots, \cD_{i+\frac1{10}\log n}^c\right) \leq \P\left(\cD_i^c,\cD_{i+1}^c,\ldots, \cD_{i+\frac1{10}\log n}^c\;\Big|\; \cB'\right) + \P\left({\cB'}^c\right) \leq 2 n^{-10}.
\]  Taking a union bound over all $i$ we have that
\[
\P\left(\exists i\,:\; \cD_i^c,\cD_{i+1}^c,\ldots, \cD_{i+\frac1{10}\log n}^c\right) \leq  n^{-9}.
\]
We have thus arrived at the following: with probability at least $1 -  n^{-9}$, for every $v\in \Z/n \Z$ there exists a block of $\log^2 n$ consecutive vertices whose histories are killed before $\tsups$ within distance $\frac15 \log^3 n$ on both the right and the left, implying the existence of the decomposition and completing the lemma.
\end{proof}

When the event $\mathcal{A}$ holds we will assume that there is some canonical choice of the $W_i$'s.  We set
\begin{equation}\label{eq-Vi-det}
V_i = \fsup(W_i, \tsups,\tpluss)\,.	
\end{equation}
 On the event that both $\cA$ and $\cB$ hold, the sets $V_i$ are disjoint, and satisfy
\begin{equation}\label{eq:VProperties}
\min_{i,i'} d(V_i,V_{i'}) \geq \tfrac12 \log^2 n \quad \hbox{ and } \quad \max_i \diam(V_i) \leq 2 \log^3 n\,.
\end{equation}
We will make use of Lemma~3.3 from \cite{LS3}, a special case of which is the following.
\begin{lemma}[\cite{LS3}]\label{l:updateSuppInq}
For any $0\leq s \leq t$ and any set of vertices $W$ we have that
\[
\big \| \P_{x_0}\left(X_t(W) \in \cdot\right) - \pi|_W \big  \|_\tv \leq \E \left[ \Big\| \P_{x_0}\left(X_s(\fsup(W,s,t)) \in \cdot\right) - \pi|_{\fsup(W,s,t)} \Big \|_\tv  \right]\, .
\]
\end{lemma}
\noindent Using this result, we have that
\[
\big \| \P_{x_0}\left(X_{\tpluss} \in \cdot\right) - \pi \big  \|_\tv \leq \E \left[ \Big\| \P_{x_0}\left(X_{\tsups}(\mbox{$\bigcup_i V_i$}) \in \cdot\right) - \pi|_{\cup_i V_i} \Big \|_\tv  \right]\, .
\]

\subsection{Coupling with product measures}
On the event $\cA\cap \cB$ we  couple  $X_{\tsups}(\bigcup_i V_i)$ and $\pi|_{\bigcup_i V_i}$ with product measures.  Since the $V_i$'s depend only on the updates along the interval $[\tsups,\tpluss]$ and are independent of the dynamics up to time $\tsups$ we will treat the $V_i$ as fixed deterministic sets satisfying~\eqref{eq:VProperties}.  Let $(\pi^{(1)},\ldots,\pi^{(m)})$ be a product measure of $m$ copies of $\pi$.  Then, by the exponential decay of correlation of the one-dimensional Ising model,
\begin{equation}\label{eq:piProdCouple}
\left\| (\pi^{(1)}|_{V_1},\ldots,\pi^{(m)})|_{V_m} - \pi\big|_{\bigcup_i V_i} \right\|_\tv \leq n^{-10}\,.
\end{equation}
Next, let $X^{(1)}_t,\ldots,X^{(m)}_t$ be $m$ independent copies of the dynamics up to time $\tsups$.  Define the event
\[
\cE=\biggl \{ \max_{v\in \cup_i V_i} \max_{\substack{0\leq s \leq \tsups\\ \fsup(v,s,\tsups) \neq \emptyset}} |v - \fsup(v,s,\tsups)| \leq \tfrac1{10} \log^2 n \biggr\}\,,
\]
and for each $1\leq j \leq m$ define the analogous event
\[
\cE^{(j)}=\biggl \{ \max_{v\in \cup_i V_i} \max_{\substack{0\leq s \leq \tsups\\ \fsup^{(j)}(v,s,\tsups) \neq \emptyset}} |v - \fsup^{(j)}(v,s,\tsups)| \leq \tfrac1{10} \log^2 n \biggr\}\,,
\]
where $\fsup^{(j)}$ is the support function for the dynamics $X^{(j)}_t$.  From Lemma~\ref{l:spread}, together with a union bound, we infer that
\begin{equation}\label{eq:cESpread}
\P\left(\cE\right) = \P(\cE^{(j)}) \geq 1 - O(n^{-10})\,.
\end{equation}
Let $\tilde{X}_t$ denote $X_t$ conditioned on $\cE$ and, similarly, let $\tilde{X}^{(j)}_t$ denote $X^{(j)}_t$ conditioned on $\cE^{(j)}$.  Then
\[
\left\|\P(\tilde{X}^{(j)}_{\tsups}(V_j) \in \cdot) - \P(X^{(j)}_{\tsups}(V_j) \in \cdot) \right\|_\tv \leq \P(\cE^{(j)}) \leq n^{-10},
\]
and so
\[
\left\|\P\left((\tilde{X}^{(1)}_{\tsups}(V_1),\ldots,\tilde{X}^{(m)}_{\tsups}(V_m)) \in \cdot\right) - \P\left((X^{(1)}_{\tsups}(V_1),\ldots,X^{(m)}_{\tsups}(V_m)) \in \cdot\right) \right\|_\tv  \leq  n^{-9}\,.
\]
Now, since the laws of the $\tilde{X}_{\tsups}(V_i)$ for distinct $i$ depend on disjoint sets of updates, they are independent and equal in distribution to $\tilde{X}^{(i)}_{\tsups}(V_i)$, hence
\[
(\tilde{X}^{(1)}_{\tsups}(V_1),\ldots,\tilde{X}^{(m)}_{\tsups}(V_m))  \stackrel{d}{=} (\tilde{X}_{\tsups}(V_1),\ldots,\tilde{X}_{\tsups}(V_m))\,.
\]
Since $\tilde{X}$ is $X$ conditioned on $\cE$,
\[
\left\|\P\left((\tilde{X}_{\tsups}(V_1),\ldots,\tilde{X}_{\tsups}(V_m)) \in \cdot\right) - \P\left((X_{\tsups}(V_1),\ldots,X_{\tsups}(V_m)) \in \cdot\right) \right\|_\tv \leq \P(\cE) \leq n^{-10}\,.
\]
Combining the previous three equations we find that
\begin{equation}\label{eq:XProductCoupling}
\left\|\P\left((X^{(1)}_{\tsups}(V_1),\ldots,X^{(m)}_{\tsups}(V_m)) \in \cdot\right) - \P\left((X_{\tsups}(V_1),\ldots,X_{\tsups}(V_m)) \in \cdot\right) \right\|_\tv  \leq 2n^{-9}\,.
\end{equation}
Thus, to show that $\big \| \P_{x_0}\left(X_{\tpluss} \in \cdot\right) - \pi \big  \|_\tv \to 0$ it is sufficient to prove that
\begin{equation}\label{eq:XProductTVBound}
\left\|\P\left((X^{(1)}_{\tsups}(V_1),\ldots,X^{(m)}_{\tsups}(V_m)) \in \cdot\right) - (\pi^{(1)}|_{V_1},\ldots,\pi^{(m)}|_{V_m}) \right\|_\tv \to 0\,.
\end{equation}

\subsection{Local $L^2$ distance}
Let $L=10$, and for each $i$ set
\[
S_i = \inf\left\{s\,:\; |\fsup(V_i, \tsups -s ,\tsups )| \leq L \right\}\,,
\]
with $S_i=0$ if $|V_i|\leq L$.

First we bound the right tail of the distribution of $S_i$.
If $|\fsup(V_i, \tsups -s ,\tsups )| > L$ then at least $L+1$ histories from $V_i$ have survived to time $\tsups -s$ and not intersected.  Hence, by equation~\eqref{eq:survivalProb},
\[
\P\left(|\fsup(V_i, \tsups -s ,\tsups )| > 10\right) \leq {|V_i| \choose L+1} e^{-(L+1) s\theta} \leq e^{-(L+1) s\theta} \log^{3(L+1)} n\,.
\]
Therefore, for $0<s<\tsups$ we see that
\begin{equation}\label{eq:SiBound}
\P\left(S_i \geq s\right) \leq e^{- s (L+1)\theta} \log^{3(L+1)} n\,.
\end{equation}
Let $\cI$ denote the event that for all $i$ we have that $S_i < \tsups$.  By~\eqref{eq:SiBound},
\begin{equation}\label{eq:cHBound}
\P\left(\cI\right) \leq e^{-(L+1)\theta  \tsups }\,  n \log^{3(L+1)} n\,,
\end{equation}
and so $\tpluss \geq \frac{2}{L \theta}\log n$ implies that $\P\left(\cI\right) \to 1$.
On the event $\cI$, we define
\[
U_i = \fsup(V_i, \tsups -S_i ,\tsups )\,.
\]
Applying Lemma~\ref{l:updateSuppInq} we have that
\begin{align}\label{eq:VtvBound}
&\left\|\P \left((X^{(1)}_{\tsups}(V_1),\ldots,X^{(m)}_{\tsups}(V_m)) \in \cdot\right) - (\pi^{(1)}|_{V_1},\ldots,\pi^{(m)}|_{V_m})\right\|_\tv\nonumber \\
&\qquad \leq  \left\|\P\left((X^{(1)}_{\tsups-S_1}(U_1),\ldots,X^{(m)}_{\tsups-S_m}(V_m)) \in \cdot\right) - (\pi^{(1)}|_{U_1},\ldots,\pi^{(m)}|_{U_m})\right\|_\tv\,.
\end{align}

\begin{lemma}\label{l:UiTVDist}
There exists $C=C(\beta)>0$ such that, for every $|U_i| \leq L$ and $0\leq S_i< \tsups$,
\begin{align*}
\left\|\P_{x_0}\left((X^{(i)}_{\tsups-S_i}(U_i) \in \cdot \;\Big|\; U_i,S_i\right) -  \pi^{(i)}|_{U_i} \right\|_\tv &\leq
\begin{cases}
C  \tsups \exp\left[-(\tsups-S_i)  \min\{2\theta,2-\theta\}\right]  &x_0 = x^{\alt}\,,\\
\noalign{\medskip}
C  \tsups \exp\left[-(\tsups-S_i)  \min\{2\theta,1\}\right]  &x_0 = x^{\blt}\,.
\end{cases}
\end{align*}
\end{lemma}
\begin{proof}
We will consider the case of $x^{\alt}$, the proof for $x^{\blt}$ follows similarly.
Let $R_i$ denote the first time the history coalesces to a single point:
\[
R_i = \inf\left\{r: |\fsup(U_i, \tsups -S_i-r ,\tsups -S_i )| \leq 1 \right\}\,,
\]
with the convention $R_i = \tsups -S_i$ if $|\fsup(U_i, 0 ,\tsups -S_i )| \geq 2$. By equation~\eqref{eq:survivalProb},
\[
\P\left(R_i > r \mid U_i,S_i\right) \leq \binom{L}2 e^{-2 r\theta}.
\]
Denote the vertex $a_i = \fsup(U_i, \tsups -S_i-R_i ,\tsups -S_i )$. By  Lemmas~\ref{l:singlePoint} and~\ref{l:updateSuppInq} we have that
\begin{align}
\left\|\P\left(X^{(i)}_{\tsups-S_i}(U_i) \in \cdot\mid U_i,S_i\right) - \pi^{(i)}|_{U_i} \right\|_\tv &\leq \E\left[ \left\|\P\left(X^{(i)}_{\tsups-S_i-R_i}(a_i) \in \cdot \mid U_i,S_i\right) - \pi^{(i)}|_{a_i} \;\Big|\; U_i,S_i \right\|_\tv\right] \nonumber\\
 &\leq \E\left[ e^{-(2 - \theta)(\tsups-S_i-R_i)} \;\big|\; U_i, S_i\right]\,.
 \label{eq:singlePtBound}
\end{align}
We estimate the right hand side as follows:
\begin{align}
\E\left[ e^{-(2 - \theta)(\tsups-S_i-R_i)} \;\big|\; U_i, S_i\right] &\leq \sum_{k=1}^{\ceil{\tsups -S_i}} \P\left(R_i \in (k-1,k)\right) e^{-(2 - \theta)(\tsups-S_i-k)}\nonumber\\
&\leq \sum_{k=1}^{\ceil{\tsups -S_i}} {L \choose 2} e^{-2 (k-1)\theta} e^{-(2 - \theta)(\tsups-S_i-k)}\nonumber\\
&\leq C  \tsups e^{-(\tsups-S_i)  \min\{2\theta,2-\theta\}} \,,
\end{align}
where the final inequality follows by taking the maximal term in the sum.  This, together with~\eqref{eq:singlePtBound}, completes the proof of the lemma.
\end{proof}
We now appeal to the $L^1$-to-$L^2$ reduction developed in~\cite{LS1,LS3}.  Recall that the $L^2$-distance on measures is defined as
\[
\left\|\mu - \pi \right\|_{L^2(\pi)} = \Big(\sum_x \Big|\frac{\mu(x)}{\pi(x)} -1 \Big|^2\pi(x)\Big)^{1/2}\,,
\]
and set
\begin{align}
  \label{eq-product-m-def}
  \ltwo_t = \sum_{i=1}^{m} \left\|\P_{x_0}[ (X^{(i)}_{\tsups-S_i}(U_i) \in \cdot \mid U_i,S_i] -  \pi^{(i)}|_{U_i} \right\|^2_{L^2(\pi^{(i)}|_{U_i})}\,.
\end{align}
By~\cite{LS3}*{Proposition~7},
\begin{equation}\label{eq:L1L2reduction}
\left\|\P\left((X^{(1)}_{\tsups-S_1}(U_1),\ldots,X^{(m)}_{\tsups-S_m}(V_m)) \in \cdot\right) - (\pi^{(1)}|_{U_1},\ldots,\pi^{(m)}|_{U_m})\right\|_\tv \leq \sqrt{\ltwo_t}.
\end{equation}
We are now ready to prove the upper bound for the main theorem.

\begin{proof}[\textbf{\emph{Proof of Theorem~\ref{mainthm-trans}, Upper bound}}]
Again we focus on the case of $x^{\alt}$.  Set
\[
\tpluss =  \frac{1}{(4-2\theta)\wedge 4\theta}\log n + \left(\kappa + \frac{3L+6}{(4-2\theta)\wedge 4\theta}  \right)\log\log n\,.
\]
With this choice of $\tpluss$ we have that $\P(\cI^c) \to 0$ and so, by equations~\eqref{eq:XProductTVBound}, \eqref{eq:VtvBound} and~\eqref{eq:L1L2reduction}, it is sufficient to show that
\begin{equation}\label{eq:UBoundCond}
\E \left[\ltwo_t \one_{\cI} \right]\to 0\, .
\end{equation}
Since each vertex is either plus or minus with probability that is uniformly bounded below by $\frac{e^{-2\beta}}{e^{-2\beta} + e^{2\beta}}$, given any choice of conditioning on the other vertices, we have that
\[
\min_{U_i} \min_{x\in \{\pm 1\}^{U_i}} \pi|_{U_i}(x) \geq \left( \frac{e^{-2\beta}}{e^{-2\beta} + e^{2\beta}} \right)^L.
\]
Comparing the $L^1$ and $L^2$ bounds we have that for any measures $\mu$ and set $U_i$,
\begin{align*}
\left\|\mu|_{U_i} - \pi|_{U_i}\right\|^2_{L^2(\pi|_{U_i})} &= \sum_x \frac1{\pi|_{U_i}(x)} \Big| \mu|_{U_i}(x) - \pi|_{U_i}(x)\Big|^2\\
&\leq 2^L \left( \frac{e^{-2\beta} + e^{2\beta}}{e^{-2\beta}} \right)^L \max_{x\in \{\pm 1\}^{U_i}} \Big|\mu|_{U_i}(x) - \pi|_{U_i}(x)\Big|^2 \\
& \leq  2^L \left( \frac{e^{-2\beta} + e^{2\beta}}{e^{-2\beta}} \right)^L  \Big\|\mu|_{U_i} - \pi|_{U_i}\Big\|^2_{\tv} \, .
\end{align*}
Thus, by Lemma~\ref{l:UiTVDist},
\begin{align*}
\E \left[\ltwo_t \one_{\cI}\right] &\leq \E \bigg[ 2^L \left( \frac{e^{-2\beta} + e^{2\beta}}{e^{-2\beta}} \right)^L \sum_{i=1}^m \left\|\P_{x_0}[ (X^{(i)}_{\tsups-S_i}(U_i) \in \cdot \mid U_i,S_i] -  \pi^{(i)}|_{U_i} \right\|^2_\tv \bigg]\nonumber\\
&\leq  2^L \left( \frac{e^{-2\beta} + e^{2\beta}}{e^{-2\beta}} \right)^L n \E \bigg[\left(C  \tsups e^{-(\tsups-S_i)  \min\{2\theta,2-\theta\}} \right)^2\bigg] \nonumber\\
&\leq C'(\beta)   e^{-\tsups  \min\{4\theta,4-2\theta\}} n \log^2 n\, \E\left[ e^{  \min\{4\theta,4-2\theta\} S_i}\right]\\
& = C'(\beta)  \left( \log n \right)^{-(3L+4)} \,\E \left[e^{  \min\{4\theta,4-2\theta\} S_i}\right]\, ,
\end{align*}
for some $C'(\beta)$.  Finally, by equation~\eqref{eq:SiBound}
\begin{align*}
\E \left[e^{  \min\{4\theta,4-2\theta\} S_i} \right]&\leq \sum_{k=1}^\infty e^{  \min\{4\theta,4-2\theta\} k} \P\left(S_i \in (k-1,k)\right)  \\
&\leq \sum_{k=1}^\infty e^{  \min\{4\theta,4-2\theta\} k}  e^{- (k-1) (L+1)\theta} \log^{3(L+1)} n  = O(\log^{3L+3} n)\,.
\end{align*}
Combining the previous two inequalities implies that $\E \ltwo_t \one_{\cI}\to 0$ and hence we have that
\[
\big \| \P_{x^{\alt}}\left(X_{\tpluss} \in \cdot\right) - \pi \big  \|_\tv \to 0\,,
\]
as required. The proof for $x^{\blt}$ follows similarly for the choice of
\[
\tpluss =  \frac{1}{2\wedge 4\theta}\log n + \left(\kappa + \frac{3L+6}{2\wedge 4\theta}  \right)\log\log n\,.\qedhere
\]
\end{proof}

\section{Lower bounds}

In order to establish the lower bound we will analyze two separate test functions.  First, in order to analyze our test functions, we establish the following decay of correlation bound.

\begin{lemma}\label{l:decorrelation}
Let $V_1,V_2 \subset \Z/ n \Z$ such that $d(V_1,V_2) \geq \log^2 n$ and let $f_i: \{\pm 1\}^{V_i} \to \R$ be functions with $\|f_i \|_\infty \leq 1$.  Then for any initial condition $x_0$ and time $t$ we have that
\[
\Cov_{x_0}(f_1(X_t(V_1)),f_2(X_t(V_2))) = O( n^{-5}).
\]
\end{lemma}
\begin{proof}
We will prove the result by showing that $Y_i = f_i(X_t(V_i))$ can be approximated locally.  Let $V_i^+ = \{v:d(v,V_i) \leq \frac1{10} \log^2 n\}$ and so the $V^+_i$ are disjoint.  Let $\cJ_i$ denote the sigma-algebra of generated by updates in $V_i^+$ and set $\widehat{Y}_i = \E_{x_0} [Y_i \mid \cJ_i]$.
Since the $V^+_i$ are disjoint the $\widehat{Y}_i$ depend on independent updates and so are independent.
Let
\[
\cG = \bigg\{ \max_{\substack{0\leq s \leq t\\ \fsup(v,s,t) \neq \emptyset}} |v - \fsup(v,s,t)| \geq \tfrac1{10} \log^2 n\bigg\}
\]
be the event in Lemma~\ref{l:spread}.  On the event $\cG$, the random variables $Y_i$ are completely determined by the initial condition and the updates in $V_i^+$ and so $Y_i I(\cG) = \widehat{Y}_i I(\cG)$.  Thus,
\[
\left|\E_{x_0}[Y_1 Y_2] - \E[\widehat{Y}_1 \widehat{Y}_2]\right| \leq \E_{x_0}[2 I(\cG^c)] = O(n^{-10})\,.
\]
and hence
\[
\Cov_{x_0}(Y_1, Y_2) = \Cov_{x_0}(Y_1, Y_2) - \Cov_{x_0}(\widehat{Y}_1, \widehat{Y}_2) =  O(n^{-10})\,,
\]
which completes the proof.
\end{proof}
Since the above bound is uniform in $t$ by taking $t$ to infinity we get the result for $X$ given by the stationary measure as well.

\subsection{Autocorrelation test functions}

The magnetization test function achieves, at least up to an additive constant, the mixing time from the all-plus initial condition, which is asymptotically the worst-case (see~\cite{LS4}).  In this light it is natural to consider test functions for $x^{\alt}$ and $x^{\blt}$ based on the autocorrelation, $\sum_{i=1}^n X_t(i) x_0(i)$.  This can be seen as a special case of a test function based on conditional expectations,
\[
R_{x_0,t}(X) = \sum_{i=1}^n X(i) \E_{x_0} [X_t(i)].
\]
Because of the special structure of the histories  as a killed random walk the expectation has the following useful representation.  Let $P_t$ be the semigroup of a continuous-time rate-1 simple random walk on $\Z/n\Z$.  Then by the killed random walk representation we have that
\[
\E_{x_0} [X_t(i)] = e^{-\theta t} (P_{(1-\theta)t} x_0)(i)
\]
The eigenvectors of $P_t$ are $e^{2\pi ikx}$ with eigenvalues $1-\cos(2\pi k\theta)$ for $k\in\{0,\ldots,n-1\}$.  Since the simple random walk is reversible with uniform stationary distribution we can write an orthonormal basis of real eigenvectors $\eta_k$  with eigenvalues $\lambda_k$.  Not that both $x^{\alt}$ and $x^{\blt}$ are eigenvectors of $P_t$ with eigenvalues $2$ and $1$ respectively and in fact $2$ is the largest eigenvalue. We first give a condition for the chain to \emph{not} be sufficiently mixed starting from $x_0$.
\begin{lemma}\label{l:lowerBoundAuto}
If for a sequence of initial conditions $x_0$ and time points $t$ we have
\[
\lim_{n\to\infty}\frac{e^{2\theta t}\log^{3} n } {\|P_{(1-\theta)t} x_0\|_2} = 0\,,
\]
then
\[
\lim_{n\to\infty} \left\|\P_{x_0}\left(X_t \in \cdot\right) - \pi \right\|_\tv = 1\,.
\]
\end{lemma}
\begin{proof}
Let $Y$ be distributed according to the stationary distribution.  Then by symmetry,
\[
\E \left[R_{x_0,t}(Y)\right] = 0\,,
\]
while
\[
\E_{x_0} \left[R_{x_0,t}(X_t) \right]= \sum_{i=1}^n (\E_{x_0} [X_t(i)])^2 = e^{-2t\theta}\left\|P_{(1-\theta)t} x_0\right\|_2^2.
\]
To estimate the variance, observe that
\begin{align*}
\var_{x_0} \left(  R_{x_0,t}(X_t) \right) &= \sum_{i=1}^n  \sum_{j=-n/2+1}^{n/2}  \Cov\left(X_t(i) \E_{x_0} \left[X_t(i)\right],\, X_t(i+j) \E_{x_0} \left[X_t(i+j)\right]\right)\\
&\leq e^{-2t\theta}\sum_{j=-n/2+1}^{n/2} \sum_{i=1}^n  \big|(P_{(1-\theta)t} x_0)(i)\big|\, \big|(P_{(1-\theta)t} x_0)(i+j) \big| \, \big|\Cov(X_t(i) ,X_t(i+k))\big|\,.
\end{align*}
By Lemma~\ref{l:decorrelation}, this is at most
\begin{align*}
& e^{-2t\theta}\sum_{j=-\log^2 n}^{\log^2 n} \sum_{i=1}^n  \big|(P_{(1-\theta)t} x_0)(i) \big| \, \big |(P_{(1-\theta)t} x_0)(i+j) \big|\\
&\qquad + O(n^{-10}) e^{-2t\theta}\sum_{j=-n/2+1}^{n/2} \sum_{i=1}^n  \big|(P_{(1-\theta)t} x_0)(i) \big| \, \big |(P_{(1-\theta)t} x_0)(i+j) \big|\\
&\leq O(\log^2 n) e^{-2t\theta}\big \| P_{(1-\theta)t} x_0 \big \|_2^2 \, ,
\end{align*}
where the final inequality follows by the rearrangement inequality.  Since Lemma~\ref{l:decorrelation} also applies to the stationary distribution, we further have
\[
\var \left(  R_{x_0,t}(Y) \right) = O(\log^2 n)e^{-2t\theta} \big \| P_{(1-\theta)t} x_0 \big \|_2^2 \, .
\]
Our test function considers the set $A = \big \{x\in \{\pm 1\}^{\Z/n\Z} : R_{x_0,t}(x) \geq \frac12 e^{-2\theta t} \| P_{(1-\theta)t} x_0 \|_2^2 \big \}$.  Therefore, by Chebyshev's inequality,
\[
\P_{x_0}\left( X_t \in A^c \right) \leq \frac{\var_{x_0} (  R_{x_0,t}(X_t) )}{\left(\E_{x_0}[ R_{x_0,t}(X_t)] - \frac12 e^{-2\theta t} \| P_{(1-\theta)t} x_0 \|_2^2  \right)^2} = O\left(\frac{e^{2\theta t} \log^2 n}{ \| P_{(1-\theta)t} x_0 \|_2^2}\right) \, ,
\]
and so by the assumption of the lemma $\P_{x_0}\left( X_t \in A \right) \to 1$.  Similarly,
\[
\P\left( Y \in A \right) \leq \frac{\var (  R_{x_0,t}(Y) )}{\left(\frac12 e^{-2\theta t} \| P_{(1-\theta)t} x_0 \|_2^2\right)^2 } = O\left(\frac{e^{2\theta t} \log^2 n}{ \| P_{(1-\theta)t} x_0 \|_2^2}\right) \, ,
\]
so $\P\left( Y \in A \right) \to 0$ which completes the lemma.
\end{proof}

We can now establish Proposition~\ref{p:autoLowerB}, giving a lower bound for any deterministic initial condition.
\begin{proof}[\textbf{\emph{Proof of Proposition~\ref{p:autoLowerB}}}]
Writing $x_0 = \sum_j b_j \eta_j$ we have that
\[
\left\| P_{t} x_0 \right\|_2^2 = \bigg\| \sum_j b_j \eta_j e^{-\lambda_j t}\bigg\|_2^2 = \sum_j b_j^2  e^{-2\lambda_j t} \geq e^{-4 t} \sum_j b_j^2 = e^{-4 t} \left\| x_0\right\|_2^2 = e^{-4 t} n \, ,
\]
where the inequality follows from the fact that all the eigenvalues are bounded by 2.  Thus,
\[
\frac{e^{2\theta \tpluss}\log^{3} n } {\|P_{(1-\theta)\tpluss} x_0\|_2} \leq \frac{e^{(4-2\theta) \tpluss}\log^{3} n } { n} \leq \frac1{\log n}\,,
\]
and so, by Lemma~\ref{l:lowerBoundAuto}, we have that $\left\|\P_{x_0}\left(X_{\tpluss} \in \cdot\right) - \pi \right\|_\tv \to 1$, as claimed.
\end{proof}
This gives the right bound in the case of $x^{\alt}$ since it is an eigenvector of eigenvalue 2.  For $x^{\blt}$ we get a stronger lower bound.  Since it has eigenvalue 1,
\[
\left\| P_{t} x^{\blt} \right\|_2^2 = \left\| e^{- t} x^{\blt}\right\|_2^2 = e^{-2 t} \left\| x^{\blt}\right\|_2^2 = e^{-2 t} n \, .
\]
So, taking $\tpluss = \frac{1}{2}\log n - 8\log\log n$,
\[
\frac{e^{2\theta \tpluss}\log^{3} n } {\|P_{(1-\theta)\tpluss} x_0\|_2} = \frac{e^{2 \tpluss}\log^{3} n } { n} \leq \frac1{\log n}\,,
\]
and hence by Lemma~\ref{l:lowerBoundAuto} we have that
\begin{equation}\label{eq:bltUBound}
\|\P_{x^{\blt}}\left(X_{\tpluss} \in \cdot\right) - \pi \|_\tv \to 1\,.
\end{equation}

\subsection{Hamiltonian test functions}
The alternating initial condition $x^{\alt}$ is an extreme value for the Hamiltonian and measuring its convergence to stationarity gives another test of convergence.  Such test functions were  studied in~\cite{LS4} to analyze the a random annealed initial condition.  To treat $x^{\alt}$ and $x^{\blt}$ in a unified manner, consider the function $R:\{\pm 1\}^{\Z/n\Z} \to \R$ given by
\[
R(X) = \sum_{i=1}^{n/4} X(4i) X(4i+1)\,.
\]
For every $x_0$ and $t$ we have that, by Lemma~\ref{l:decorrelation},
\begin{align}\label{eq:RVariance}
\var_{x_0} (  R(X_t) ) &= \sum_{i=1}^{n/4}  \sum_{j=1}^{n/4}   \Cov\left(X(4i) X(4i+1),\, X(4j) X(4j+1)]\right)= O(n \log^2 n) \, .
\end{align}
If $Y$ is taken from the stationary distribution by taking a limit as $t\to \infty$, then we also have that $\var (  R(Y) ) = O(n \log^2 n)$.  Let $\sH$ denote the set of all histories of the vertices from time $\tpluss$, and  consider $\E_{x_0}[X_{\tpluss}(i)X_{\tpluss}(i')\mid \sH]$.  If the histories of $i$ and $i'$ merge then $X_{\tpluss}(i)$ and $X_{\tpluss}(i')$ must take the same value and $\E_{x_0}[X_{\tpluss}(i)X_{\tpluss}(i')\mid \sH]=1$.  If the histories do not merge and at least one is killed before reaching time 0 then it is equally likely to be $\pm 1$ so $\E_{x_0}[X_{\tpluss}(i)X_{\tpluss}(i')\mid \sH]=0$.  Thus, the boundary condition can only play a role when both histories survive to time 0 and do not merge, as captured by the event
\[
\cK_{i,i'} = \left\{ |\fsup(i,0,\tpluss)|=|\fsup(i',0,\tpluss)|=1, \, \fsup(i,0,\tpluss)\neq \fsup(i',0,\tpluss) \right\}.
\]
Let $Y$ be an independent configuration distributed as $\pi$  and let $\E_{\pi}$ denotes the expectation started from the stationary measure.  Then
\begin{align}
\E_{\pi}\left[X_{\tpluss}(i)X_{\tpluss}(i+1) \one_{\cK_{i,i+1}}\mid \sH\right] &= \E_{\pi}\left[X_{0}(\fsup(i,0,t))X_{0}(\fsup(i+1,0,t))\one_{\cK_{i,i+1}}\mid \sH\right] \nonumber\\
&= \E[Y(\fsup(i,0,t))Y(\fsup(i+1,0,t)) \one_{\cK_{i,i+1}}\mid \sH] \geq 0 \, ,
\end{align}
as the ferromagnetic Ising model is positively correlated.  In a graph with two vertices connected by an edge, the correlation of spins of the Ising model can be found to be $\tanh \theta$. Correlations are monotone in the edges of the graph, so for neighboring vertices in $\Z/n\Z$ we have $\E[Y(i)Y(i+1)] \geq \tanh \theta >0$.
It was shown in the proof of Theorem~6.4 of~\cite{LS4} that
\begin{align*}
\P\left(\fsup(i,0,t)=\{i\},\,\fsup(i+1,0,t)=\{i+1\}\right)  \geq c_1 t^{-2} e^{-2\theta t}\,,
\end{align*}
and so
\begin{align}\label{eq:Wstationary}
\E_{\pi}\left[X_{\tpluss}(i)X_{\tpluss}(i+1)\one_{\cK_{i,i+1}}\mid \sH\right]
&\geq \tanh(\theta) \P\left(\fsup(i,0,\tpluss)=\{i\},\,\fsup(i+1,0,\tpluss)=\{i+1\}\right)\nonumber\\
&\geq c_1 \tanh(\theta) \tpluss^{-2} e^{-2\theta \tpluss} \, .
\end{align}
We will compare this bound with the behavior under the initial conditions $x^{\alt}$ and $x^{\blt}$.
\begin{claim}
For $x_0 \in \{x^{\alt},x^{\blt}\}$ and $i\equiv 0\pmod 4$ we have that
\[
\E_{x_0}\left[X_{\tpluss}(i)X_{\tpluss}(i+1)\one_{\cK_{i,i+1}}\mid \sH\right] \leq 0.
\]
\end{claim}
\begin{proof}
We  first treat the case of $x^{\alt}$.  Let $Z_1(t)$ and $Z_2(t)$ be independent  rate-($1-\theta$) continuous-time simple random walks with initial conditions $Z_1(0)=i$ and $Z_2(0)=i+1$.  Let $T$ denote the first time the walks hit each other and $W(t) = x^{\alt}(Z_1(t))x^{\alt}(Z_2(t))$. By the killed random walk representation of the histories, we have that
\begin{align*}
\E_{x^{\alt}}\left[X_{\tpluss}(i)X_{\tpluss}(i+1)\one_{\cK_{i,i+1}}\mid \sH\right] = e^{-2\theta} \E \left[x^{\alt}(Z_1(\tpluss))x^{\alt}(Z_2(\tpluss))\one_{T>\tpluss}\right] = e^{-2\theta} \E \left[W(\tpluss)\one_{T>\tpluss}\right].
\end{align*}
Note that $W(t)$ is itself a Markov chain with state space $\{\pm 1\}$ and transition rate $2(1-\theta)$, and so
\begin{equation}\label{eq:RWsignEstimate}
\E[W(t+s)\mid W(s)] = e^{-4(1-\theta)} W(s)\,.
\end{equation}
Thus, since $W(0)=-1$ by the definition of $x^{\alt}$, and $W(T)=1$, applying~\eqref{eq:RWsignEstimate} we get
\begin{align*}
\E \left[W(\tpluss)\one_{T>\tpluss}\right] &= \E \left[W(\tpluss)\right] - \E \left[W(\tpluss)\one_{T\leq\tpluss}\right] = -e^{-4(1-\theta)\tpluss} - \E \left[\one_{T\leq\tpluss} \E \left[W(\tpluss) \Big | T \right]  \right]\\
& = -e^{-4(1-\theta)\tpluss} - \E \left[\one_{T\leq\tpluss} e^{-4(1-\theta)(\tpluss-T)}  \right] \leq 0\,.
\end{align*}
Hence, $\E_{x^{\alt}}[X_{\tpluss}(i)X_{\tpluss}(i+1)\one_{\cK_{i,i+1}}\mid \sH] \leq 0$.

For $x^{\blt}$, the process $x^{\blt}(Z_1(t))x^{\blt}(Z_2(t))$ is again a Markov chain but with transition rate $1-\theta$.  The requirement that $i$ is a multiple of 4 was chosen to ensure that $x^{\blt}(Z_1(0))x^{\blt}(Z_2(0))=-1$. The argument is otherwise unchanged.
\end{proof}
Combining Lemma~\ref{eq:RWsignEstimate} with equation~\eqref{eq:Wstationary}, we obtain that
\begin{align*}
&\E_\pi[X_{\tpluss}(i)X_{\tpluss}(i+1)] - \E_{x^{\alt}}[X_{\tpluss}(i)X_{\tpluss}(i+1)]\\
 & \qquad= \E_\pi[X_{\tpluss}(i)X_{\tpluss}(i+1)\one_{\cK_{i,i+1}}] - \E_{x^{\alt}}[X_{\tpluss}(i)X_{\tpluss}(i+1)\one_{\cK_{i,i+1}}] \geq c_1 \tanh(\theta) \tpluss^{-2} e^{-2\theta \tpluss} \, ,
\end{align*}
and thus
\begin{align}\label{eq:RComparison}
\E[R(Y)] - \E_{x_0}[R(X_{\tpluss})] \geq  c_1 \tanh(\theta) \tpluss^{-2} e^{-2\theta \tpluss} \frac{n}{4} \, .
\end{align}
We are now ready to prove the second lower bound.
\begin{lemma}\label{l:HamiltonianLowerB}
Set
\[
\tpluss = \tfrac1{4\theta} \tpluss - \tfrac{5}{\theta} \log \log n\,.
\]
For $x_0 \in \{x^{\alt},x^{\blt}\}$ we have
\[
\lim_{n\to\infty} \|\P_{x_0}[X_{\tpluss} \in \cdot] - \pi \|_\tv = 1\,.
\]
\end{lemma}
\begin{proof}
Denote by $A$ the event
\[
A = \big \{x\in \{\pm 1\}^{\Z/n\Z} : R(x) \geq \tfrac12 (\E[R(Y)] + \E_{x_0}[R(X_{\tpluss})]) \big \}\,.
\]
By Chebyshev's inequality and equations~\eqref{eq:RVariance} and~\eqref{eq:RComparison}\,
\[
\P_{x_0}(X_{\tpluss}\in A) \leq \frac{\var_{x_0}(R(X_{\tpluss}))}{\left(\frac12(\E[R(Y)] - \E_{x_0}[R(X_{\tpluss})])\right)^2} = O\left( \frac{n \log^2 n}{\tpluss^{-4} e^{-4\theta \tpluss}n^2} \right) \to 0 \, ,
\]
and similarly
\[
\P(Y\in A^c) \leq \frac{\var(R(Y))}{\left(\frac12(\E[R(Y)] - \E_{x_0}[R(X_{\tpluss})])\right)^2} \to 0 \, .
\]
Hence, $\|\P_{x_0}(X_{\tpluss} \in \cdot) - \pi \|_\tv \to 1$, as claimed.
\end{proof}

\begin{proof}[\textbf{\emph{Proof of Theorem~\ref{mainthm-trans}, Lower bound}}]
The case of $x^{\alt}$ follows from combining Proposition~\ref{p:autoLowerB} and Lemma~\ref{l:HamiltonianLowerB}, while the lower bound for $x^{\blt}$ follows from equation~\eqref{eq:bltUBound} and Lemma~\ref{l:HamiltonianLowerB}.
\end{proof}

\subsection*{Acknowledgements} We thank Yuval Peres for helpful discussions.

\end{document}